\documentclass[12pt,a4paper]{amsart}
\usepackage{amssymb}
\usepackage[shortlabels]{enumitem}
\usepackage[utf8]{inputenc}
\usepackage{float}

\usepackage{tikz-cd}
\usepackage{hyperref}

\newtheorem{theorem}{Theorem}[section]
\newtheorem{thmx}{Theorem}

\newtheorem{lemma}[theorem]{Lemma}

\newtheorem{question}[theorem]{Question}		
	
\theoremstyle{definition}

\theoremstyle{remark}
				
\newtheorem{remark}[theorem]{Remark}

\newcommand{\Z}{\mathbb{Z}}

\newcommand{\E}{\mathcal{E}}
\newcommand{\Ef}{\mathcal{E}_{free}}
\newcommand{\G}{\mathfrak{G}}
\newcommand{\K}{\mathcal{K}}
\newcommand{\Sym}{\operatorname{Sym}}
\def\hght{\mathrm{ht}}
\def\aut{\mathrm{aut}}

\def\id{\mathrm{id}}

\begin{document}

\title[Permutation modules]{Realization of permutation modules \\via Alexandroff spaces}

\author[C. Costoya]{Cristina Costoya}
\address{CITMAga, Departamento de Matem\'aticas,
Universidade de Santiago de Compostela, 15705-A Coru{\~n}a, Spain.}
\email{cristina.costoya@usc.es}

\author[R. Gomes]{Rafael Gomes}
\address{Departamento de \'Algebra, Geometr\'{\i}a y Topolog\'{\i}a, Universidad de M\'alaga,
Campus de Teatinos, s/n, 29071-M\'alaga, Spain}
\email{rmgomes@uma.es}

\author[A. Viruel]{Antonio Viruel}
\address{Departamento de \'Algebra, Geometr\'{\i}a y Topolog\'{\i}a, Universidad de M\'alaga,
Campus de Teatinos, s/n, 29071-M\'alaga, Spain}
\email{viruel@uma.es}

\thanks{The first author was supported by MINECO (Spain) grant PID2020-115155GB-I00.
The second and third authors were supported by MINECO (Spain) grant PID2020-118753GB-I00. The second author was also supported by Fundação para a Ciência e Tecnologia (Portugal) grant 2021.04682.BD}


\begin{abstract}

We raise the question of the realizability of permutation modules in the context of Kahn's realizability problem for abstract groups and the $G$-Moore space problem. Specifically, given a finite group $G$, we consider a collection $\{M_i\}_{i=1}^n$ of finitely generated $\Z G$-modules that admit a submodule decomposition on which $G$ acts by permuting the summands. Then we prove the existence of connected finite spaces $X$ that realize each $M_i$ as its $i$-th homology, $G$ as its group of self-homotopy equivalences $\E(X)$, and the action of $G$ on each $M_i$ as the action of $\E(X)$ on $H_i(X; \Z)$.

\end{abstract}

\maketitle

\section{Introduction}

Realizability questions represent a classical branch of problems in Algebraic Topology. Two widely known examples are the \textit{realizability problem for abstract groups} proposed by D. Kahn \cite{kahn}, which seeks to characterise the groups that appear as the group of self-homotopy equivalences of  simply-connected spaces, and the \textit{$G$-Moore space problem}, $G$ a group, proposed by Steenrod \cite[Problem 51]{lashof}, which asks for $\Z G$-modules that emerge as the homology of  simply-connected $G$-Moore spaces. The former has been solved for finite groups \cite{CV},  and remains open in its full generality. Nonetheless, for non-simply-connected spaces, it has been been successfully addressed for arbitrary groups through Alexandroff spaces \cite[Theorem 1.3, Theorem 1.4]{chocano}. Concerning the latter problem, it has been demonstrated that not all $\Z G$-modules are realisable \cite{carlsson}.

In this paper, we raise the following question that  combines both of the above mentioned problems, and can thought of as the dual of \cite[Question 1.3]{CV6}:


\begin{question} \label{Question}
For a given group $G$, decide for which $\Z G$-modules $M$ there exists a topological space $X$ such that
\begin{enumerate}[label={\rm (\roman{*})}]
    \item $H_\ast(X;\Z) \cong M$;
    \item the group of self-homotopy equivalences $\E(X)$ is isomorphic to $G$;
    \item  the action of $G$ on $M$ is equivalent to the canonical action of $\E(X)$ on $H_\ast(X;\Z)$.
\end{enumerate}
\end{question}

In this paper, we provide a positive answer to this question in the particular case where $G$ is a finite group and  $M$ is a finitely generated \textit{generalized permutation module},
that is, there exists a submodule decomposition  $M=\bigoplus\limits_{v\in V} M_v$, where $V$ is a finite set,  on which $G$ acts as a permutation group on the summands as described by some permutation representation $\rho\colon G\to \Sym(V)$ \cite[pag.\ 686]{Kar}.  We prove the following result:
\begin{thmx}\label{thm:main}
Let $n$ be a positive integer, and $G$ be a finite group. If $\{M_i\}_{i=1}^n$ is a collection of finitely generated permutation $\Z G$-modules, then there exist infinitely many non homotopically equivalent connected finite spaces $X$ such that
\begin{enumerate}[label={\rm (\roman{*})}]    \item\label{thm:main_1} $\Ef(X)=\E(X)\cong G$;
    \item\label{thm:main_2} $H_i(X;\Z)=0$ for $i>n$;
    \item\label{thm:main_3} the pair $(\E(X),H_i(X;\Z))$ is equivalent to $(G, M_i)$ for $i=1,\ldots, n$.
\end{enumerate}

\end{thmx}

The strategy to establish Theorem \ref{thm:main} mirrors the approach developed by the first and third authors to provide an affirmative answer to the finite case of Kahn's realizability problem \cite{CV}. The initial step involves realizing permutation representations within the context of graphs using results from \cite{CMV1, Frucht1}. Subsequently, we proceed to construct topological spaces that translate this solution to the topological setting. In contrast to the rational homotopy techniques employed in \cite{CV}, we adopt the use of Alexandroff spaces in this current work.

We recall that \textit{Alexandroff spaces} are topological spaces for which the intersection of an arbitrary number of open sets remains open. It is well known that there exists an isomorphism of categories between the category of $T_0$ Alexandroff spaces, with continuous maps, and the category of partially ordered sets with order-preserving maps. Henceforth, we shall assume that all spaces are $T_0$, and we shall make no distinction between an Alexandroff space and a poset.
The theory of Alexandroff spaces, particularly concerning finite spaces, has gained considerable popularity in recent years, thanks to May's REU programs \cite{may} and Barmak's work \cite{barmak}. Additionally, as we have already mentioned, in \cite{chocano} and \cite{chocano1}, Chocano, Morón, and Ruiz del Portal have effectively employed Alexandroff spaces in this context of realizability problems.  We draw inspiration from their ideas, and previous work from the first and third author.

This paper is organized as follows. Section \ref{section_graphs} focuses on realizing permutation representations as automorphisms of graphs. Moving on to Section \ref{sect:acyclic}, we establish a connection between graphs and their incidence posets, which enables us to translate the results from the previous section into the topological setting. Specifically, we demonstrate the realization of permutation representations as the automorphisms of acyclic finite spaces with arbitrary large height.  Finally, in Section \ref{sect:proof_main} we prove Theorem \ref{thm:main} by combining the results from Section \ref{sect:acyclic} and the realizability results in \cite{chocano1}.


\section{Graphs and permutation representations} \label{section_graphs}

We consider simple, undirected, and connected graphs. For a graph $\mathcal{G}$, we denote $V(\mathcal{G})$ as the vertex set and $E(\mathcal{G})$ as the edge set. The aim of this section is to refine certain results from \cite[Section 3]{CMV}, and prove the following theorem:

\begin{theorem}\label{thm:graph_realizing_permutation}
Let $G$ be a finite group, and let $\rho:G\to \Sym(V)$ be a not necessarily faithful permutation representation where $V$ is a finite non-empty set. Then there exist infinitely many non-isomorphic finite connected graphs $\mathcal{G}$ such that the following hold:
\begin{enumerate}[label={\rm (\roman{*})}]
    \item\label{thm:graph_1} $V\subset V(\mathcal{G})$ and $V$ is invariant under the $\aut(\mathcal{G})$-action;
    \item\label{thm:graph_2} Every vertex in $\mathcal{G}$ has degree at least $2$;
    \item\label{thm:graph_3} $G \cong \aut(\mathcal{G})$;
    \item\label{thm:graph_4} The restriction $G\cong\aut(\mathcal{G})\to \Sym(V)$ is precisely $\rho$.
\end{enumerate}
\end{theorem}

Let us briefly recall some relevant results that will be needed in the proof of  Theorem \ref{thm:graph_realizing_permutation}.
Following \cite[Section 5.1]{HIK},  for $n$ a positive integer, the \textit{cartesian product} of graphs $\mathcal{G}_1, \mathcal{G}_2, \ldots, \mathcal{G}_n$,  denoted by $\mathcal{G}_1\square \mathcal{G}_2\square\ldots\square \mathcal{G}_n$, has
vertex set $\prod_{j=1}^n V(\mathcal{G}_j)$, and vertices $(x_1,\ldots, x_n)$ and $(y_1,\ldots, y_n)$ are adjacent if $\{x_i,y_i\}\in E(\mathcal{G}_j)$ for exactly one index $1\leq i\leq n$ and $x_j=y_j$ for each index $j\ne i$.

A graph $\mathcal{G}$ is \textit{prime} (with respect to the cartesian product) if it is nontrivial and cannot be expressed as $\mathcal{G}=\mathcal{G}_1\square \mathcal{G}_2$ where $\mathcal{G}_i$ is not isomorphic to $\mathcal{K}_1$, the graph of one isolated vertex. For any connected finite graph $\mathcal{G}$, there exists a unique factorization (up to isomorphism and the order of the factors) of $\mathcal{G}$ as a cartesian product of prime graphs: $\mathcal{G} = \mathcal{P}_1 \square \ldots \square \mathcal{P}_r$ \cite[Theorem 6.6]{HIK}. Furthermore, $\aut(\mathcal{G})$ is fully described in terms of the groups $\aut(\mathcal{P}_i)$ and the isomorphism type of the graphs $\mathcal{P}_i$, $i=1, \ldots, r$ \cite[Theorem 6.10]{HIK}. A graph $\mathcal{G}$ is said  \textit{rigid} if $\aut(\mathcal{G})=\{1\}$.

We now prove that there are infinitely many prime rigid graphs:
 \begin{lemma}\label{lem:rigidos_primos}
There exist infinitely many non-isomorphic prime finite connected simple graphs $\mathcal{P}$ such that $\aut(\mathcal{P})=\{1\}$.
\end{lemma}
\begin{proof}
Let $\mathcal{G}$ be a connected rigid finite graph and let $\mathcal{G}=\mathcal{P}_1\square\ldots\square \mathcal{P}_r$ be its unique prime factorization. By \cite[Theorem 6.10]{HIK}, and since $\aut(\mathcal{G})=\{1\}$, then $\mathcal{P}_i\cong \mathcal{P}_j$ if and only if $i=j$ and moreover $\aut(\mathcal{P}_i)=\{1\}$ for every $i=1,\ldots,r$.

Then, a finite number of connected prime graphs $\mathcal{P}_i$ such that $\aut(\mathcal{P}_i)=\{1\}$ would indeed produce only a finite number of isomorphism types of graphs $\mathcal{G}$ with $\aut(\mathcal{G})=\{1\}$. However, according to \cite{Frucht1}, there exist infinitely many finite connected graphs with trivial automorphism groups. Therefore, there must exist infinitely many prime graphs $\mathcal{P}$ such that $\aut(\mathcal{P})=\{1\}$
\end{proof}

\begin{proof}[Proof of Theorem \ref{thm:graph_realizing_permutation}]

According to  \cite[Theorem 3.1]{CMV}, there exists a  simple and undirected graph $\mathcal{G}$ that satisfies \ref{thm:graph_1}, \ref{thm:graph_3} and \ref{thm:graph_4} which is built upon the  binary $I$-system $\G$ given in \cite[Definition 3.5]{CMV}. As $V(\G)=G\sqcup V$, we obtain that $\G$ is finite. Moreover, every vertex in $G\subset V(\G)$ is connected with every vertex in $V\subset V(\G)$, hence $\G$ is connected.   As described in \cite[Section 3.1]{CMV1}, $\mathcal{G}$ is obtained out of $\G$ by an arrow replacement operation. Therefore, since $\G$ is finite, it has a finite number of labels $I$, and the connected asymmetric graphs $\mathcal{R}_i$, $i\in I$, involved in the replacement operation, can be chosen finite and connected. This shows that the graph $\mathcal{G}$ can be chosen to be connected and finite.


We now prove that $\mathcal{G}$ gives rise to an infinite family of graphs satisfying conditions \ref{thm:graph_1} to \ref{thm:graph_4} above. Let $\mathcal{G}=\mathcal{P}_1\square\ldots\square \mathcal{P}_r$ represent its unique prime factorization, and let $\{\mathcal{Q}_j\}_{j\in J}$ denote the collection of connected prime non-isomorphic rigid graphs that are not divisors of $\mathcal{G}$. In other words, $\aut(\mathcal{Q}_j)={1}$, and $\mathcal{Q}_j\not\cong \mathcal{P}_i$ for any $j$ and $i$. Due to Lemma \ref{lem:rigidos_primos}, considering that $i=1,\ldots, r$ is a finite number of indices,  $\{Q_j\}_{j\in J}$ needs to be an infinite collection.

Let us define $\mathcal{G}_j=\mathcal{G}\square \mathcal{Q}_j$, and by selecting a vertex $q_j\in V(\mathcal{Q}_j)$, we can establish the identification $V=V\times \{q_j\}\subset  V(\mathcal{G})\times V(\mathcal{Q}_j)=V(\mathcal{G}_j)$. With this in place, we claim that the collection $\{\mathcal{G}_j\}_{j\in J}$ is an infinite set of non-isomorphic graphs satisfying conditions \ref{thm:graph_1} -- \ref{thm:graph_4} above.

Indeed, since both $\mathcal{G}$ and $\mathcal{Q}_j$ are finite and connected, it follows that $\mathcal{G}_j$ is so. Furthermore, when  $j\ne j'$, the prime decomposition of $\mathcal{G}_j$ and $\mathcal{G}_{j'}$ are different, hence $\mathcal{G}_j\not\cong \mathcal{G}_{j'}$. Additionally, as established in \cite[Theorem 6.10]{HIK}, $\aut(\mathcal{G}_j)=\aut(\mathcal{G})\times\aut(\mathcal{Q}_j)=G$, since condition \ref{thm:graph_3} holds  for $\mathcal G$ and $\mathcal{Q}_j$ is rigid. This implies that condition \ref{thm:graph_3} also holds for $\mathcal{G}_j$, with automorphisms acting coordinate-wise.

Consequently, for any $\psi\in\aut(\mathcal{G}_j)$, we have $\psi=(\psi_1,\operatorname{Id})\in \aut(\mathcal{G})\times\aut(\mathcal{Q}_j)$, leading to $\psi(V)=\psi_1(V)\times\{q_j\}=V$ and therefore \ref{thm:graph_1} holds. Moreover, the restriction $G\cong\aut(\mathcal{G}_j)\to \Sym(V)$ is precisely the restriction $G\cong\aut(\mathcal{G})\to \Sym(V)$, i.e.\ $\rho$, satisfying \ref{thm:graph_4}.

Finally, we must confirm condition \ref{thm:graph_2}. Since both $\mathcal{G}$ and $\mathcal{Q}_j$ are connected, each vertex in $\mathcal{G}$ and $\mathcal{Q}_j$ has a minimum degree of $1$. As vertex degrees add up in cartesian products, \cite[Proposition 5.1]{HIK}, it follows that every vertex in $\mathcal{G}_j$ has a degree of at least $2$.
\end{proof}

\section{Acyclic finite spaces and permutation representations}\label{sect:acyclic}

In this section, we establish an analogue of Theorem \ref{thm:graph_realizing_permutation} for acyclic minimal finite spaces. A finite space $X$ is usually represented by its \textit{Hasse diagram}, a directed graph in which the vertices correspond to the points of $X$. An edge $(x, y)$ exists if $x \lneq y$ and there is no $z \in X$ such that $x \lneq z \lneq y$. Then, we say that $y$ covers $x$ or that $x$ is covered by $y$. A \textit{minimal finite space} refers to a finite space without beat points, i.e., points that are either covered by, or cover, a single point. We refer the reader to \cite[Section 1.3]{barmak} for more details.

Recall that for a given graph $\mathcal{G}$, there exists a finite space $(P_\mathcal{G},\leq_\mathcal{G})$ known as the \textit{incidence poset} of $\mathcal{G}$, which is defined as follows (see  \cite{order_poset}): $P_\mathcal{G} = V(\mathcal{G}) \cup E(\mathcal{G})$ as a set, and, for any two distinct elements $a$ and $b$ in $P_\mathcal{G}$, the relation $a \leq_\mathcal{G} b$ holds if and only if $a$ is a vertex and $b$ is an edge incident at $a$, or if $a$ and $b$ are the same vertex or edge. In other words, $(P_\mathcal{G},\leq_\mathcal{G})$ is a height $1$ finite space  (i.e.\ the longest chain has $2$ elements) constructed over $V(\mathcal{G}) \cup E(\mathcal{G})$, with the minimal (resp. maximal) elements in $(P_\mathcal{G},\leq_\mathcal{G})$ representing the vertices (resp.\ edges) of $\mathcal G$. Consequently, given that the automorphisms of a finite space must preserve the minimality (resp.\ maximality) of elements, while graph automorphisms must preserve incidence, there exists an obvious isomorphism $\aut(\mathcal{G}) \cong \aut(P_\mathcal{G})$. In this setting, Theorem \ref{thm:graph_realizing_permutation} can be reformulated as:

\begin{theorem}\label{thm:finite_realizing_permutation}
Let $G$ be a finite group, and let $\rho:G\to \Sym(V)$ be a not necessarily faithful  permutation representation, where $V$ is a finite non-empty set. Then, there exist infinitely many non-isomorphic minimal finite spaces $X$ satisfying the following conditions:
\begin{enumerate}[label={\rm (\roman{*})}]
    \item\label{thm:finite_1} $V$ is a discrete subspace within $X$, all the elements of $V$ are minimal, and $V$ is invariant under the $\aut(X)$-action;
    \item\label{thm:finite_2} The height of $X$ is $1$;
    \item\label{thm:finite_3} $G \cong \aut(X)$;
    \item\label{thm:finite_4} The restriction $G\cong\aut(X)\to \Sym(V)$ is precisely $\rho$.
\end{enumerate}
\end{theorem}

\begin{remark}
 For our specific needs, it is crucial to focus on minimal finite spaces. This is justified by the fact that  homotopy equivalences between minimal finite spaces are homeomorphisms, \cite[Corollary 1.3.7]{barmak}. Thus,  when dealing with minimal finite spaces, the group of homeomorphisms $\aut(X)$ and the group of self-homotopy equivalences $\E(X)$ are isomorphic.
\end{remark}

In the remainder of this section, we will establish a version of Theorem \ref{thm:finite_realizing_permutation} for acyclic spaces with arbitrarily large height.

\begin{theorem}\label{thm:finite_contractiible_realizing_permutation}
Let $n \geq 5$ be a positive integer. Consider $G$ a finite group, and a permutation representation $\rho:G\to \Sym(V)$, where $V$ is a finite non-empty set. Then there exist infinitely many non-isomorphic minimal finite spaces $X:=(X,\leq)$ that satisfy the following conditions:
\begin{enumerate}[label={\rm (\roman{*})}]
    \item\label{thm:finite_con_0} Every minimal element in $X$ is part of a chain of maximal length $n$. In particular, $X$ has height $n$;
    \item\label{thm:finite_con_1} The set $V$ forms a discrete subspace of $X$, with all elements of $V$ considered minimal. Additionally $V$ remains invariant under the $\aut(X)$-action;
    \item\label{thm:finite_con_2} The space $X$ is weakly contractible, thus acyclic;
    \item\label{thm:finite_con_3} We have $G \cong \E_{free}(X)=\E(X)=\aut(X)$;
    \item\label{thm:finite_con_4} The restriction $G\cong\aut(X)\to \Sym(V)$ corresponds to the representation $\rho$.
\end{enumerate}
\end{theorem}

To prove Theorem \ref{thm:finite_contractiible_realizing_permutation}, we need to introduce several lemmas. The subsequent lemma encapsulates two fundamental properties that follow straightforwardly from the continuity of the inverse function $f^{-1}$ when $f$ is a homeomorphism.

\begin{lemma}\label{facts_automorphism_finite_spaces}
Let $X$ be a finite space, and let $f$ be a homeomorphism of $X$. Then the following hold:
\begin{enumerate}[label={\rm (\roman{*})}]
\item \label{fact1_automorphism_fintie_spaces}
 $x_1 \leq x_2$ if and only if $f(x_1) \leq f(x_2)$, for all $x_1,x_2 \in X$;

\item \label{fact2_automorphism_finite_spaces}  $x_0 \leq x_1 \leq \dots \leq x_n$ is a chain (resp.\ maximal chain) in $X$ if and only if $f(x_0) \leq f(x_1) \leq \dots \leq f(x_n) $ is a chain (resp.\ maximal chain) in $X$ as well.
\end{enumerate}
\end{lemma}

We recall that for finite spaces $X$ and $Y$, the \textit{non-Hausdorff join} $X \circledast Y$ is defined as the space with the underlying set being the disjoint union $X\sqcup Y$, and its partial order established by keeping the orderings within $X$ and $Y$, while declaring $x\leq y$ for each $x\in X$ and $y\in Y$. The following lemma enumerates several useful properties associated with this operation:

\begin{lemma}
Let $X$ and $Y$ be finite spaces  with more than one point. Then: \label{lem:non_Hausdorff_join}
\begin{enumerate}[label={\rm (\roman{*})}]
    \item If both $X$ and $Y$ are minimal, then $X \circledast Y$ is also minimal; \label{lem:non_Hausdorff_joincond1}
    \item If $Y$ (or $X$) is weakly contractible, then $X \circledast Y$ is also weakly contractible; \label{lem:non_Hausdorff_joincond2}
    \item  \label{lem:non_Hausdorff_joincond3}  The height of the non-Hausdorff join satisfies $$\hght(X \circledast Y) = \hght(X) + \hght(Y) +1 ;$$
    \item $\aut(X \circledast Y) \cong \aut(X) \times \aut(Y)$. \label{lem:non_Hausdorff_joincond4}
\end{enumerate}
\end{lemma}

\begin{proof}
     Since there are no beat points in $X$ and $Y$ by hypothesis, and the non-Hausdorff join operation makes every point of $Y$ cover each point of $X$, it follows that $X \circledast Y$ has no beat points, thus satisfying property \ref{lem:non_Hausdorff_joincond1}.

    Assertion \ref{lem:non_Hausdorff_joincond2} can be deduced from the observation that the non-Hausdorff join of $X$ and $Y$ is weakly equivalent to the join of the geometric realizations of the McCord complexes of $X$ and $Y$. The McCord complex of a finite space $X$ is a simplicial complex whose vertices represent the points of $X$, and the simplices are the non-empty chains of $X$ (for more details, see \cite[Section 1.4]{barmak}). So, if we denote the McCord complex of $X$ as $K(X)$ and its geometric realization as $\vert \K(X) \vert$, then, as outlined in \cite[Remark 2.7.2]{barmak} and the initial observations in \cite[Section 2.7]{barmak}, we have the following sequence of equivalences:
\[
X \circledast Y \simeq \vert \K(X \circledast Y) \vert = \vert \K(X) \ast \K(Y) \vert \cong \vert \K(X) \vert \ast \vert \K(Y) \vert,
\]
where $\ast$ indicates the join operation for either simplicial complexes or topological spaces. If $Y$ is weakly contractible, then, according to \cite[Corollary 1.4.15]{barmak}, the geometric realization $\vert \K(Y) \vert$ becomes contractible. Consequently, we conclude that $X \circledast Y \simeq \vert \K(X) \vert \ast \{y_0\}$, which is a contractible space.


  Property \ref{lem:non_Hausdorff_joincond3} is also obvious, since the longest chain of $X \circledast Y$ is the union of the longest chain of $X$ and the longest chain of $Y$.

  Finally, to establish \ref{lem:non_Hausdorff_joincond4}, we claim that the map $\Phi: \aut(X \circledast Y) \rightarrow \aut(X) \times \aut(Y)$, defined by associating $f$ with $(f|_X, f|_Y)$, is an isomorphism of groups. This is due to the fact that $f(X)=X$ and $f(Y)=Y$. Indeed, suppose there exists an element $x \in X$ such that $f(x) \in Y$. According to Lemma \ref{facts_automorphism_finite_spaces}.\ref{fact1_automorphism_fintie_spaces}, we have $f(x) \leq f(y)$ for all $y \in Y$. Since no element of $X$ is a covering point for any element of $Y$, we deduce that $f(Y) \subsetneq Y$ which leads to a contradiction as $f$ is a bijection. Hence, it must be that $f(X)=X$, and consequently, $f(Y)=Y$. This establishes that $\Phi$ is bijective and a homomorphism of groups.

\end{proof}

We now have all the necessary ingredients in place to establish:
\begin{proof}[Proof of Theorem \ref{thm:finite_contractiible_realizing_permutation}]

Consider a minimal finite space $X_1$ from Theorem \ref{thm:finite_realizing_permutation}. Then $X_1$ already satisfies conditions \ref{thm:finite_con_1}, \ref{thm:finite_con_3}, and \ref{thm:finite_con_4}, and it satisfies condition \ref{thm:finite_con_0} when $n=1$. Our objective is to modify $X_1$ in a way that it becomes weakly contractible while keeping its group of automorphisms unchanged.

To accomplish this, we will make use of $L_1$, the finite space introduced in \cite{Riv},  whose Hasse diagram is illustrated in Figure \ref{DiagL1}. It is weakly contractible but not contractible, \cite{ottina}, with $\aut(L_1) \cong \Z/2$ and $\hght(L_1)=2$. This space was key  in \cite{chocano1} to realize groups as the group of homeomorphisms as well as group of self-homotopy equivalences of Alexandroff spaces.

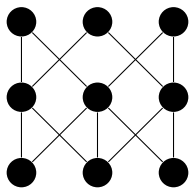
\begin{figure}[H]
 \begin{tikzpicture}
  [scale=1,auto=center,every node/.style={circle,fill=black}]
  \node (n1) at (2,0)  {};
  \node (n2) at (2,1)  {};
  \node (n3) at (2,2)  {};
  \node (n4) at (1,0)  {};
  \node (n5) at (1,1)  {};
  \node (n6) at (1,2)  {};
  \node (n7) at (0,0)  {};
  \node (n8) at (0,1)  {};
  \node (n9) at (0,2)  {};

  \foreach \from/\to in {n1/n2,n1/n5,n2/n3,n2/n6,n4/n5,n4/n2,n4/n8,n5/n7,n3/n5,n5/n9,n7/n8,n8/n9,n8/n6}
    \draw (\from) -- (\to);

\end{tikzpicture}
\caption{Hasse diagram of $L_1$} \label{DiagL1}
\end{figure}

Next, we define finite spaces $W_k$, introduced in \cite{chocano1}, recursively: set $W_1 = L_1$ and for $k > 1$,  $W_k = W_{k-1} \vee L_1$, where the wedge sum identifies a maximal point of $L_1$ that is not fixed by the non-trivial automorphism of $L_1$, with a point in the first position of a chain of maximum length in $W_{k-1}$ (which is also not fixed by the non-trivial automorphism of the copy of $L_1$ in which it lies). The Hasse diagram for $W_2$ is shown in Figure \ref{DiagW2}.

\begin{figure}[H]
 \begin{tikzpicture}
  [scale=1,auto=center,every node/.style={circle,fill=black}]

  \node (m1) at (2,0)  {};
  \node (m2) at (2,1)  {};
  \node (m3) at (2,2)  {};
  \node (m4) at (1,0)  {};
  \node (m5) at (1,1)  {};
  \node (m6) at (1,2)  {};
  \node (m7) at (0,0)  {};
  \node (m8) at (0,1)  {};
  \node (m9) at (0,2)  {};
  \node (m10) at (2,-1) {};
  \node (m11) at (2,-2) {};
  \node (m12) at (3,0) {};
  \node (m13) at (3,-1) {};
  \node (m14) at (3,-2) {};
  \node (m15) at (4,0) {};
  \node (m16) at (4,-1) {};
  \node (m17) at (4,-2) {};

  \foreach \from/\to in {m1/m2,m1/m5,m2/m3,m2/m6,m4/m5,m4/m2,m4/m8,m5/m7,m3/m5,m5/m9,m7/m8,m8/m9,m8/m6,m1/m10,m1/m13,m10/m11,m10/m14,m11/m13,m13/m14,m14/m16,m15/m16,m15/m13,m16/m12,m16/m17,m17/m13,m12/m10}
    \draw (\from) -- (\to);
\end{tikzpicture}
\caption{Hasse diagram of $W_2=L_1 \vee L_1$}\label{DiagW2}
\end{figure}
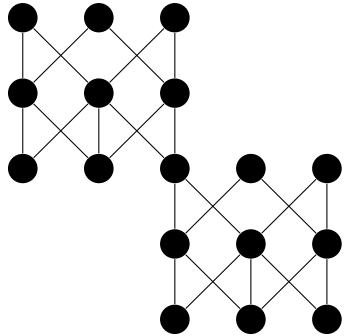

This construction ensures, in the first place,  that $W_k$ remains a weakly contractible minimal space: $W_k$ has no beat points as the wedge does not decrease the number of points covering or covered by each point, and $W_k =W_{k-1} \vee L_1$ with $W_1=L_1 \simeq \ast$,  so by induction we have $W_k \simeq \ast$.

In the second place,  for $k \geq 2$, we show that $\aut(W_k)= \{ 1 \}$ by induction. To prove that $\aut(W_2)=\{1\}$, observe that any $f \in \aut(W_2)$, must fix the gluing point, which is the only point in third position within chains of length $4$ in $W_2$. By Lemma \ref{facts_automorphism_finite_spaces}.\ref{fact1_automorphism_fintie_spaces} we deduce that minimal points not covered by the gluing point are sent to minimal points not covered by the gluing point, so $f$ takes each copy of $L_1$ to itself. By the choice of the gluing point, the restriction of $f$ to each $L_1$ is the identity and therefore $f$ is the identity.
    The same argument shows that if $f \in \aut(W_k)$, then $f(L_1)=L_1$ and $f|_{L_1}=\id_{L_1}$. Using the induction hypothesis, $\aut(W_{k-1})= \{1 \}$, we conclude $\aut(W_{k})= \{1 \}$.

Finally,  we claim that the height of $W_k$ is $2k$:  by induction suppose that $\hght(W_{k-1})=2(k-1)$ and note  that $\hght(W_k)=\hght(W_{k-1})+\hght(L_1)$, since the wedge identifies a minimal point with a maximal point.

Thus, by Lemma \ref{lem:non_Hausdorff_join}, $X= X_1 \circledast W_k$,  for $k \geq 2$, is a weakly contractible minimal finite space of height $\hght(X) =\hght(X_1) +\hght(W_k)+ 1 =2k+2$, whose automorphism group is $\aut(X_1) \times \aut(W_k) \cong \aut(X_1) \cong G$. Hence, it satisfies all the conditions from the theorem, as condition \ref{thm:finite_con_1} follows from the fact that $G$ acts as the identity on $W_k$ and, as previously on $X_1$.

It is worth noting that we have so far obtained minimal finite spaces with even heights. To extend our results to finite spaces with odd heights that fulfill the conditions of Theorem \ref{thm:finite_contractiible_realizing_permutation}, we introduce a modification. We replace the existing $W_k$ with a new finite space 
$\widetilde W_k = W_{k-1} \vee L_1$ for $k \geq 2$ where the identification is between a maximal point of $L_1$, which is not fixed by the non-trivial automorphism of $L_1$, and a point situated in the second position of a chain of maximal length within $W_{k-1}$ (likewise not fixed by the non-trivial automorphism of the copy of $L_1$). This is represented in Figure \ref{DiagTildeW2}.

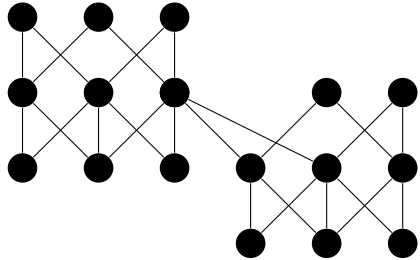
\begin{figure}[H]
 \begin{tikzpicture}
  [scale=1,auto=center,every node/.style={circle,fill=black}]

  \node (m1) at (2,0)  {};
  \node (m2) at (2,1)  {};
  \node (m3) at (2,2)  {};
  \node (m4) at (1,0)  {};
  \node (m5) at (1,1)  {};
  \node (m6) at (1,2)  {};
  \node (m7) at (0,0)  {};
  \node (m8) at (0,1)  {};
  \node (m9) at (0,2)  {};
  \node (m10) at (5,-1)  {};
  \node (m20) at (5,0)  {};
  \node (m30) at (5,1)  {};
  \node (m40) at (4,-1)  {};
  \node (m50) at (4,0)  {};
  \node (m60) at (4,1)  {};
  \node (m70) at (3,-1)  {};
  \node (m80) at (3,0)  {};
  \node (m90) at (2,1)  {};

  \foreach \from/\to in {m1/m2,m1/m5,m2/m3,m2/m6,m4/m5,m4/m2,m4/m8,m5/m7,m3/m5,m5/m9,m7/m8,m8/m9,m8/m6,m10/m20,m10/m50,m20/m30,m20/m60,m40/m50,m40/m20,m40/m80,m50/m70,m30/m50,m50/m90,m70/m80,m80/m90,m80/m60}
    \draw (\from) -- (\to);
\end{tikzpicture}
\caption{Hasse diagram of $\widetilde{W}_2=W_1 \vee L_1$}\label{DiagTildeW2}
\end{figure}

The resulting finite space $\widetilde{W}_k$ is minimal, weakly contractible, has no non-trivial automorphisms, and a height of $2k-1$. Then, $X =X_1 \circledast \widetilde W_k $, for $k \geq 2$,  has height $\hght(X) =\hght(X_1) +\hght(\widetilde W_k)+ 1 =2k+1$.

\end{proof}


\section{Realization of permutation modules} \label{sect:proof_main}


\begin{proof}[Proof of Theorem \ref{thm:main}]
Following the notation in the hypotheses, let  $M_i=\bigoplus\limits_{v\in V_i} (M_i)_v$ be the submodule decomposition on which $G$ acts as a permutation group on the summands as described by some permutation representation $\rho_i\colon G\to \Sym(V_i)$.

Using \cite[Corollary 1.4]{chocano1},  there exist topological spaces (that are finite and minimal), denoted as $Y_i^v$ for each $1 \leq i \leq n$  and $v \in V_i$, satisfying that $\E(Y_i^v)=\aut(Y_i^v)=\{1\}$ and $$\widetilde{H}_j(Y_i^v;\Z)=\begin{cases} 0, &j\ne i.\\ (M_i)_v, &j=i.
\end{cases}$$

Furthermore,  for each $1 \leq i \leq n$ and $v \in V_i$, a maximal point $y_i^v$ within  $Y_i^v$ is selected,  ensuring that it belongs to a chain of maximum length $l_i^v:=\hght(Y_i^v).$

Finally, for any $1 \leq i \leq n$, if  two elements $v, v' \in V_i$  are in the same $G$-orbit, we make the choice to set $Y_i^v = Y_i^{v'}$ and $y_i^v = y_i^{v'}.$

On the other hand, by Theorem \ref{thm:finite_contractiible_realizing_permutation}, there exists a minimal finite space $Z$, which is weakly contractible and satisfies the following conditions:
\begin{itemize}
    \item $\ell:=\hght(Z)> l_i^v$ for all $1 \leq i \leq n$ and $v \in V_i$, and every minimal element of $Z$ is in a chain of maximum length $\ell$;
    \item $V= \coprod \limits_{i=1}^n V_i$ is a discrete subspace of $Z$, invariant under the $\aut(Z)$-action,  where all elements of $V$ are minimal in $Z$;
    \item $G \cong \aut(Z) =  \E(Z); $
    \item  The restriction $G \cong \aut(Z) \rightarrow \Sym(V)$ is precisely $\rho=\bigoplus \limits_{i=1}^n \rho_i.$
\end{itemize}


Now we define the finite space $X$  as the following wedge sum
$$ X :=Z \bigvee_{\substack{v \in V_i \\ i=1}}^n Y_i^v $$
where  we identify the element $v \in V_i  \subset V  \subset  Z$ with the fixed maximal point $y_i^v \in Y_i^v$.

We need to show that $X$ is a minimal finite space,  $\widetilde{H}_i(X;\Z)=M_i$, $\E(X)=\aut(X) \cong \aut(Z) \cong G$ and  finally, that the action of $G$ on $M_i$ is equivalent to the action of $\E(X)$ on $H_i(X;\Z)$, for all $1 \leq i \leq n$.

Checking that $X$ has no beat points is straightforward.  Indeed, as  $Z$ and all of the $Y_i^v$ are minimal spaces,   by identifying $v \in V_i$ with $y_i^v \in Y_i^v$, neither the number of points covered by $v \sim y_i^v$ decreases nor does the number of points covering $v \sim y_i^v$.  Hence $X$ is a minimal finite space.

Furthermore, by direct application of the wedge axiom of homology, we obtain that, for $1 \leq i \leq n$, $\widetilde{H}_i(X; \Z)=\bigoplus \limits_{v \in V_i} (M_i)_{v}=M_i$.

Now, for any $f \in \aut(X)$ we claim that $f|_Z \in \aut(Z)$ and that this restriction completely determines $f$.

First, we are going to prove that $f(V)=V$ by induction argument. Recall that $l_i^v=\hght(Y_i^v)$ and let $$l=\max\{l_i^v :  v \in V_i, \,1 \leq i \leq n\}.$$ By construction  $l<\ell= \hght(Z)$. Define $J_k=\{ (i, v) : l_i^v\geq l-k\}$. We start by noting that if $(j, w)\in J_0$,  i.e.\ $l_j^w=l$, then $Y_j^w$ has maximum height, and $w \in V_j$ is in position $l+1$ of a chain of maximum length $l+\ell $.  Therefore, by Lemma \ref{facts_automorphism_finite_spaces} \ref{fact2_automorphism_finite_spaces},  $f(w) = w'$ where $w' \in V_{j'} $ for some $1 \leq j'\leq n$ with $(j', w') \in J_0$.
Now, assume that for $(j, w)\in J_{k-1}$, we have that $f (w) = w'$,  with $(j', w') \in  J_{k-1}$ for some $1 \leq j'\leq n.$ Let $(j, w)\in J_k\smallsetminus J_{k-1}$, that is  $l_{j}^w = l-k$ and
there is a maximal chain of length $l-k+\ell$,
\begin{equation}\label{eq:chain}x_0<x_1<\ldots<x_{l-k}=w<\ldots <x_{l-k+\ell}
\end{equation}
such that $x_{l-k}=w$ is the only point of this chain that belongs to $V$, and $x_{l-k+\ell}\in Z$ is a maximal point. Now,  since $\ell>l \geq l_i^v$, for $1\leq i\leq n$ and $v \in V_i$,  then  the subchain
$$f(w)=f(x_{l-k})<\ldots <f(x_{l-k+\ell})$$
must end in $Z$, thus $f(x_{l-k+\ell})$ is necessarily a maximal point in $Z$. Therefore, since $\ell=\hght(Z)$, there exist $l-k\leq r < l-k+\ell$ and $1\leq j'\leq n$ such that $f(x_r) = w'\in V_{j'}$.  We have to prove that $r={l-k}$, i.e. $f(w)=w'$, and $(j', w') \in J_k\smallsetminus J_{k-1}.$ Now,   $w' \in V_{j'}$ is in a maximal chain of maximal possible length $l_{j'}^{w'} + \ell$
$$x'_0<x'_1<\ldots<x'_{l_{j'}^{w'}}=w'<\ldots <x'_{l_{j'}^{w'} + \ell}$$
while it is also in the maximal chain
$$
f(x_0)<f(x_1)<\ldots<f(x_r)=w'= x'_{l_{j'}^{w'}}<\ldots <x'_{l_{j'}^{w'} + \ell}$$
of length $r+\ell$ and since $l_{j'}^{w'} + \ell$ was the maximal possible length for a chain containing $w'$, we get that $l_{j'}^{w'}\geq r \geq l-k$. Suppose that $r > l-k$,  that is  $ (j', w')\in J_{k-1}$. Then, by the induction hypothesis there exists $ (j'', w'')\in J_{k-1}$ such that $f(w'')=w'$. Since $f$ is a homeomorphism, $w''=x_r$ and therefore $x_r\in V$, while $x_{l-k}=w$ was the only point of the chain \eqref{eq:chain} belonging to $V$, which contradicts $r>l-k$. Hence $r=l-k$, and by a similar argument $l_{j'}^{w'}= l-k$, that is $(j', w') \in J_k \smallsetminus J_{k-1}$
and $f(V) = V$ as claimed.

We now prove that $f(Z)=Z$. If $z \in Z \smallsetminus V$ is a minimal point of $X$ (i.e., a minimal point of $Z$ that is not in $V$), then $z$ is the first element in a maximal chain of length $\ell$. By Lemma \ref{facts_automorphism_finite_spaces} \ref{fact2_automorphism_finite_spaces}, $f(z)$ is also the first element in a maximal chain of length $\ell$. This means $f(z) \notin Y_i^v$ for all $i$ and $v \in V$, as if that was the case $f(z)$ would be the first element in a maximal chain containing a point $w \in V$, because $\hght(Y_i^v) <\ell$ and a chain starting in some $Y_j^w$ and ending in $Z$ must pass through $V$. We would then get $f(z) \leq w$ and since $f(V)=V$, by Lemma \ref{facts_automorphism_finite_spaces} \ref{fact1_automorphism_fintie_spaces} $z \leq v$ for some $v \in V$, which contradicts $v$ being minimal in $Z$. Thus $f(z)\in Z$. Since every point in $Z$ covers either a point in $V$ or a point in $Z\smallsetminus V$ which is minimal in $X$, and all these are mapped into $Z$, then $f(Z)=Z$ by \ref{facts_automorphism_finite_spaces} \ref{fact1_automorphism_fintie_spaces}.

We now prove that $\aut(X) \cong \aut(Z)$. Since $f(Z)=Z$, and the restriction map $\aut(Z) \rightarrow \Sym(V)$ is $\rho$, we have that $f(V_i)=V_i$ for each $i$. Take $v \in V_i \subset V$ for some $i$ and assume $f(v)=w \in V_i$, that is, $v$ and $w$ are in the same $G$-orbit. Recall that $v$ is identified with a fixed maximal point $y_i^v \in Y_i^v$ and $w$ is identified with a fixed maximal point $y_i^w \in Y_i^w$, but since they are in the same $G$-orbit, we made the choice to set $(Y_i^v, y_i^v) = ( Y_j^w, y_j^w)$. Then, $f(y_i^v)=y_i^w = y_i^v$ and all the points $y \in Y_i^v$ such that $y \leq y_i^v$ are sent to $Y_i^w= Y_i^v$. Since $f(Z)=Z$, by the connectivity of $Y_i^v$ we have that $f(Y_i^v)=Y_i^w= Y_i^v$. Then, $f|_{Y_i^v}$ is an automorphism of $Y_i^v$ which is the identity as $\aut(Y_i^v)= \{ 1 \}$. This is valid for each $1 \leq i \leq n$, therefore the map $f \in \aut(X)$ is completely determined by its restriction to $Z$. Thus, $\aut(X) \cong \aut(Z)$.

Furthermore, since the action of $G \cong \aut(Z)$ on $Z$ permutes the points of $V$ according to the representation $\rho=\bigoplus \limits_{i=1}^n\rho_i:G \rightarrow \Sym(V)$ and the automorphisms of $X$ are determined by their restriction to $Z$, then the action of $G \cong \aut(X)$ on $X$ permutes the subspaces $Y_i^v$, $v \in V_i$, as prescribed by the representation $\rho_i:G \rightarrow \Sym(V_i)$ for each $i$ such that $1 \leq i \leq n$. Therefore, since $\widetilde H_\ast(Y_i^v; \Z)=(M_i)_v$ and $H_i(X; \Z) \cong \bigoplus \limits_{v \in V_i} H_i(Y_i^v; \Z)$, we can see that $\E(X)$ acts on $H_i(X;\Z) \cong M_i= \bigoplus \limits_{v \in V_i} (M_i)_v$ by permuting the summands as described by the representation $\rho_i:G \rightarrow \Sym(V_i)$. This is precisely how $G$ acts on $M_i$, hence condition \ref{thm:main_3} is shown.

\end{proof}

\end{document}